\documentclass[12pt]{article}%
\usepackage{amsmath}
\usepackage{amsfonts}
\usepackage{amssymb}
\usepackage{graphicx}%
\providecommand{\U}[1]{\protect\rule{.1in}{.1in}}
\newtheorem{theorem}{Theorem}

\newtheorem{corollary}[theorem]{Corollary}

\newtheorem{lemma}[theorem]{Lemma}

\newtheorem{proposition}[theorem]{Proposition}
\newtheorem{remark}[theorem]{Remark}

\newenvironment{proof}[1][Proof]{\noindent\textbf{#1.} }{\ \rule{0.5em}{0.5em}}
\begin{document}

\title{ Complete minimal discs in Hadamard manifolds}
\author{Jaime Ripoll\\UFRGS\\Instituto de Matem\'{a}tica \\Av. Bento Gon\c{c}alves 9500\\91540-000 Porto Alegre-RS\\Brazil\\jaime.ripoll@ufrgs.br
\and Friedrich Tomi\\Universit\"{a}t Heidelberg\\Mathematisches Institut\\Im Neuhenheimer Feld 288\\69120 Heidelberg\\Germany\\tomi@mathi.uni-heidelberg.de}
\date{}
\maketitle

\begin{abstract}
Using the classical approach we show the existence of disc type solutions to
the asymptotic Plateau problem in certain Hadamard manifolds which may have
arbitrarily strong curvature and volume growth.

\end{abstract}

\section{Introduction}

\qquad The study of complete minimal submanifolds of negatively curved
Riemannian manifolds was initiated by Anderson when he showed that each closed
submanifold of the sphere at infinity of the hyperbolic $n$-space is the
boundary (relative to the geometric compactification of the hyperbolic space)
of an area minimizing variety \cite{An},\ \cite{An2}. In the sequel a
considerable number of related and more general results has appeared in the
literature (see the survey by Baris Coskunuzer \cite{C1} which give an account
of the state of art till 2009). We mention in particular the work of Victor
Bangert and Urs Lang \cite{BL} where they prove existence results for complete
minimizing varieties in manifolds the metric of which is bi-Lipschitz
equivalent with a Riemannian metric of sectional curvature bounded between two
negative constants. As far as we can see all previous authors use the
Geometric Measure Theory approach or they work in the graph setting (see
\cite{DLR}, \cite{GS}, \cite{NS}, \cite{L}, \cite{RT2} for works using the
graph approach).

In this paper we use the methods of the classical Plateau problem to prove the
existence of minimal discs in certain $n$-dimensional Hadamard manifolds with
a prescribed Jordan curve at infinity as their boundary. The technique of
representing surfaces by mappings which we apply here allows the control of
the topological type, however at the cost of a restriction to two dimensional
surfaces. A different approach was undertaken by B. Coskunuzer in the case of
the three dimensional hyperbolic space, even for surfaces of constant mean
curvature \cite{C2}.

The metrics which we consider in this paper do in general not fall into the
class of metrics in the paper of Bangert and Lang mentioned above since they
need not satisfy a growth condition for the volume of geodesic balls as it
were the case for metrics which are bi-Lipschitz equivalent to a metric of
sectional curvature bounded below by a constant. On the other hand, our
metrics are restricted in a different respect: In a sense made precise below
(see Theorem \ref{main1}) they must be comparable with a rotational metric, in
particular they are bi-Lipschitz equivalent with such a metric.

A further important feature of our result lies in the fact that we need not
require the existence of convex barriers at infinity which seem to be of
fundamental importance in the previous papers. To that regard, we mention the
recent work of the first author of this paper together with Jean-baptiste
Casteras and Ilkka Holopainen \cite{CHR}. We are able to dispense with the
convexity at infinity since we introduce coordinates for our manifold in which
the mappings approximating the solution surface have bounded norm in the
Sobolev space $H^{1,2}$. Thus we may use the concept of boundary values of
such functions and interpret the boundary condition for the limit surface in
the sense of Sobolev spaces. In the course of this approximation process we
are confronted with a possible energy concentration phenomenon at the boundary
of the surfaces. In the absence of convex barriers at infinity we exclude this
possibility through a blow up argument. The concept of convexity at infinity
which may or may not hold in our case is discussed in greater detail below. A
precise description of our results follows now.

Let $N^{n},$ $n\geq3,$ be a Hadamard manifold, that is, $N$ is a connected,
simply connected, complete, $n-$dimensional Riemannian manifold such that
$K_{N}\leq0,$ where $K_{N}$ is the supremum of the sectional curvatures of $N$
at any plane of the tangent space at any point of $N$. For the sake of
simplicity, we may assume that $N^{n}$ is $C^{\infty}$ smooth. Recall that the
asymptotic boundary $\partial_{\infty}N$ of $N$ is defined as the set of all
equivalence classes of unit speed geodesic rays in $N$; two such rays
$\gamma_{1},\gamma_{2}:\left[  0,\infty\right)  \rightarrow M$ are equivalent
if $\sup_{t\geq0}d\left(  \gamma_{1}(t),\gamma_{2}(t)\right)  <\infty$, where
$d$ is the Riemannian distance in $N.$ The so called \emph{geometric}
compactification $\overline{N}$ of $N$ is then given by $\overline{N}%
:=N\cup\partial_{\infty}N,$ endowed with the cone topology. It is well known
that $\overline{N}$ is homeomorphic to the closed unit ball of $\mathbb{R}%
^{n}$ (see \cite{EO} or \cite{SY}, Ch. 2). For any subset $S\subset N$, we
define $\partial_{\infty}S=\partial_{\infty}N\cap\overline{S}.$

Setting $r(x)=d(x,o),$ $x\in N,$ where $o$ is a fixed point in $N$, $d$ the
Riemannian distance and%

\begin{equation}%
\begin{array}
[c]{c}%
k^{+}(s)=\sup\left\{  K\left(  \operatorname{grad}r(x),Y\right)  \text{
$\vert$
}r(x)=s,\text{ }Y\in T_{x}N\right\} \\
k^{-}(s)=\inf\left\{  K\left(  \operatorname{grad}r(x),Y\right)  \text{
$\vert$
}r(x)=s,\text{ }Y\in T_{x}N\right\}
\end{array}
\label{kk}%
\end{equation}
we prove:

\begin{theorem}
\label{main1}Assume that there is a continuous non-increasing negative
function $k_{0}$ defined on the interval $\left[  0,+\infty\right)  $ such that

(1)%
\[
k^{+}(x)\leq k_{0}(s)<0,\text{ }0\leq s<+\infty
\]

(2)
\[
\frac{k^{-}-k_{0}}{\sqrt{-k_{0}}}\in L^{1}\left(  \left[  0,+\infty\right)
\right)  .
\]

Then there is a Riemannian metric $\left\langle \text{ , }\right\rangle _{B}$
in the unit ball $B\subset\mathbb{R}^{n}$ such that $\left(  B,\left\langle
\text{ , }\right\rangle _{B}\right)  $ is isometric to $N$ and the asymptotic
boundary $\partial_{\infty}B$ of $B$ is identified with the topological
boundary $\partial B$ of $B.$ In this model $\left(  B,\left\langle \text{ ,
}\right\rangle _{B}\right)  $ of $N$, given a Euclidean rectifiable curve
$\Gamma\subset\partial_{\infty}B,$ there is a proper, minimal (possibly
branched) immersion $u:D\rightarrow B,$ where $D$ is the unit disc in
$\mathbb{R}^{2},$ such that $u$ belongs to the Sobolev space $H_{2}%
^{1}(D,\mathbb{R}^{n})$ and the trace of $u|\partial D$ parametrizes $\Gamma$
monotonically. In the case $n=3$ the map $u$ is an embedding of $D.$
\end{theorem}

\begin{remark}
The function $k_{0}$ above is to be thought of as the curvature of a
rotational background metric. Our interest is in the case that $k_{0}(s)$
converges to $-\infty$ when $s\rightarrow+\infty,$ so that (2) becomes the
lesser restrictive the faster $k_{0}$ converges to $-\infty.$
\end{remark}

It follows from the properties of the Sobolev trace that $\partial_{\infty
}u(D)\supset\Gamma\ $but we do not know if $\partial_{\infty}u(D)=\Gamma$
under the hypothesis of Theorem \ref{main1}. However we may conclude this
equality if one requires additionally that $N$ has some global convexity
property. This is the case if $N$ satisfies the \emph{strict convexity
condition}, as defined in \cite{RT}, namely: Given $x\in\partial_{\infty}N$
and a relatively open subset $W\subset\partial_{\infty}N$ containing $x,$
there exists a $C^{2}$-open subset $\Omega\subset\overline{N}$ such that
$x\in\operatorname*{Int}\left(  \partial_{\infty}\Omega\right)  \subset W,$
where $\operatorname*{Int}\left(  \partial_{\infty}\Omega\right)  $ denotes
the interior of $\partial_{\infty}\Omega$ in $\partial_{\infty}N,$ and
$N\setminus\Omega$ is convex.

We remark that under the assumption $K_{N}\leq-a^{2}<0,$ the strict convexity
condition is equivalent to the \emph{convex conic neighborhood condition }as
defined by H. Choi in \cite{Ch}. It is proved in \cite{RT} that if $K_{N}%
\leq-a^{2}$ then $N$ satisfies the strict convexity condition either if the
metric of $N$ is rotationally symmetric or if the sectional curvature of $N$
decays at most exponentially (Theorems 13 and 14 of \cite{RT}).

\begin{theorem}
\label{main2}Under the same hypothesis of Theorem \ref{main1} if,
additionally, $N$ satisfies the strict convexity condition then, besides the
conclusions of Theorem \ref{main1} it holds $\partial_{\infty}u(D)=\Gamma.$
This holds, in particular, if the metric of $M$ is rotationally symmetric or
if the sectional curvature of $N$ decays at most exponentially.
\end{theorem}

\section{Differential geometric preliminaries}

\begin{lemma}
\label{Fsao}Let $k\in C^{0}\left(  \left[  0,\infty\right)  \right)  ,$
$k\leq0,$ and let $F$ be the solution of the initial value problem%
\begin{equation}
F^{\prime\prime}+kF=0,\text{ }F(0)=0,\text{ }F^{\prime}(0)=1. \label{F}%
\end{equation}
Then we have:

(i) For all $s>0$ it holds
\begin{equation}
\frac{sF^{\prime}(s)}{F(s)}\geq1 \label{efe}%
\end{equation}
and, if for some constant $a>0$ the inequality $k(s)<-a^{2}$ holds for all $s$
then it follows that%
\[
\frac{F^{\prime}}{F}\geq a.
\]

(ii) If $k$ is non-increasing then $G/(sF)$ is non-increasing, too, where
$G(s)=\int_{0}^{s}F(t)dt.$

(iii) Let $F_{0}$ be a further solution of (\ref{F}) with $k_{0}$ replacing
$k,$ $k_{0}<0.$ We assume that $k_{0}$ is non-increasing and that the function%
\[
\varphi(s):=\frac{\left\vert k(s)-k_{0}(s)\right\vert }{\sqrt{-k_{0}}}%
\]
is integrable on $\left[  0,+\infty\right)  $. Then one has the estimate%
\[
e^{-C}\leq\frac{F(s)}{F_{0}(s)}\leq e^{C},\text{ }0\leq s<+\infty
\]
with
\[
C=\frac{\pi}{2}\int_{0}^{+\infty}\varphi(t)dt.
\]

\end{lemma}

\begin{proof}
(i) Using (\ref{F}) one sees that $g:=F^{\prime}/F$ solves the initial value
problem%
\begin{equation}
g^{\prime}=-k-g^{2},\text{ }g(0)=+\infty. \label{ge}%
\end{equation}

Since $k\leq0$ we have $\left(  1/g\right)  ^{\prime}\leq1$ and hence
$g(s)\geq1/s,$ proving (\ref{efe}). If moreover $k(s)\leq-a^{2}$ it follows
that%
\begin{equation}
g^{\prime}\geq a^{2}-g. \label{gln}%
\end{equation}
We clearly have $g(s)>a$ on some interval $\left[  0,s_{0}\right]  ,$
$s_{0}>0.$ We consider the maximal interval $\left[  s_{0},R\right)  $ on
which $g(s)>a.$ On this interval we get from (\ref{gln})%
\[
\frac{1}{2a}\left(  \ln\frac{g-a}{g+a}\right)  ^{\prime}=\frac{g^{\prime}%
}{g^{2}-a^{2}}\geq-1,
\]
what upon integration yields%
\[
\frac{g(s)-a}{g(s)+a}\geq\frac{g(s_{0})-a}{g(s_{0})+a}e^{-2a(s-s_{0}%
)}>0,\text{ }s_{0}\leq s<R.
\]
This shows that $R=+\infty,$ i.e. $g(s)>a$ for $s>0.$

(ii) The statement is equivalent with
\[
\left(  \ln\frac{sF}{G}\right)  ^{\prime}\geq0
\]
that is%
\[
\left(  \ln\frac{sF}{G}\right)  ^{\prime}=\frac{1}{s}+\frac{F^{\prime}}%
{F}-\frac{G^{\prime}}{G}=\frac{1}{s}+\frac{F^{\prime}G-FG^{\prime}}{FG}.
\]
We compute%
\[
\left(  F^{\prime}G-FG^{\prime}\right)  ^{\prime}=F^{\prime\prime}%
G-FG^{\prime\prime}=-kFG-FG^{\prime\prime}=-F\left(  G^{\prime\prime
}+kG\right)  .
\]
Integrating (\ref{F}) and using that $k(s)$ is non-increasing, we obtain%
\begin{align*}
0  &  =F^{\prime}(s)-1+\int_{0}^{s}k(t)F(t)dt\\
&  \geq F^{\prime}(s)-1+k(s)\int_{0}^{s}F(t)dt=G^{\prime\prime}(s)+k(s)G(s)-1.
\end{align*}
The last inequality yields%
\[
\left(  F^{\prime}G-FG^{\prime}\right)  ^{\prime}\geq-F
\]
and, upon integration and observing $G(0)=0,$ $G^{\prime}(0)=F(0)=0,$%
\[
F^{\prime}G-FG^{\prime}\geq-G^{\prime}.
\]
This leads to
\[
\left(  \ln\frac{sF}{G}\right)  ^{\prime}\geq\frac{1}{s}-\frac{1}{F},
\]
proving (ii) since $F(s)\geq s$ for all $s$ because of $F^{\prime}(0)=1$ and
$F^{\prime\prime}\geq0.$

(iii) One computes
\[
\left\vert \left(  F^{\prime}F_{0}-FF_{0}^{\prime}\right)  ^{\prime
}\right\vert =\left\vert F^{\prime\prime}F_{0}-FF_{0}^{\prime\prime
}\right\vert =\left\vert k-k_{0}\right\vert FF_{0}%
\]
leading to%
\begin{align}
\left\vert \frac{F^{\prime}(s)F_{0}(s)-F(s)F_{0}^{\prime}(s)}{F(s)F_{0}%
(s)}\right\vert  &  \leq\int_{0}^{s}\frac{\left\vert k(t)-k_{0}(t)\right\vert
}{F(s)F_{0}(s)}F(t)F_{0}(t)dt\label{mod}\\
&  \leq\int_{0}^{s}\left\vert k(t)-k_{0}(t)\right\vert \frac{F_{0}(t)}%
{F_{0}(s)}dt,\nonumber
\end{align}
since $F$ is non-decreasing. For a fixed $t$ let $f$ be the solution of the
initial value problem
\[
f^{\prime\prime}(s)+k_{0}(t)f(s)=0,\text{ }f(t)=F_{0}(t),\text{ }f^{\prime
}(t)=F_{0}^{\prime}(t).
\]
Since $k_{0}$ is non-increasing by assumption we conclude that for $s\geq t$%
\begin{align*}
F_{0}(s)  &  \geq f(s)=F_{0}(t)\cosh\left(  \sqrt{-k_{0}(t)}\left(
s-t\right)  \right) \\
&  +\frac{F_{0}^{\prime}(t)}{\sqrt{-k_{0}(t)}}\sinh\left(  \sqrt{-k_{0}%
(t)}\left(  s-t\right)  \right) \\
&  \geq F_{0}(t)\cosh\left(  \sqrt{-k_{0}(t)}\left(  s-t\right)  \right)  .
\end{align*}
Thus we get with $0<\delta<r$%
\begin{align*}
\left\vert \ln\frac{F(r)}{F_{0}(r)}-\ln\frac{F(\delta)}{F_{0}(\delta
)}\right\vert  &  =\left\vert \int_{\delta}^{r}\left(  \frac{F^{\prime}%
(s)}{F(s)}-\frac{F_{0}^{\prime}(s)}{F_{0}(s)}\right)  ds\right\vert \\
&  \leq\int_{\delta}^{r}\int_{0}^{s}\frac{\left\vert k(t)-k_{0}(t)\right\vert
}{\cosh\left(  \sqrt{-k_{0}(t)}\left(  s-t\right)  \right)  }dtds\\
&  \leq\int_{0}^{\infty}\left\vert k(t)-k_{0}(t)\right\vert \int_{t}^{\infty
}\frac{ds}{\cosh\left(  \sqrt{-k_{0}(t)}\left(  s-t\right)  \right)  }dt\\
&  =\int_{0}^{\infty}\left\vert k(t)-k_{0}(t)\right\vert \left[  \frac
{\arctan\left(  \sinh\sqrt{-k_{0}(t)}\left(  s-t\right)  \right)  }%
{\sqrt{-k_{0}(t)}}\right]  _{t}^{\infty}dt\\
&  =\frac{\pi}{2}\int_{0}^{\infty}\frac{\left\vert k(t)-k_{0}(t)\right\vert
}{\sqrt{-k_{0}(t)}}dt=:C<+\infty.
\end{align*}

Since%
\[
\frac{F(\delta)}{F_{0}(\delta)}\rightarrow1\text{ as }\delta\rightarrow0
\]
the statement follows.
\end{proof}

\bigskip

In the following we consider a $n-$dimensional Hadamard manifold $N,$ whose
Riemannian metric is denoted by $\left\langle \text{ },\text{ }\right\rangle
.$ We fix a point $o\in N$ and consider the distance function $r(x)=d(x,o)$ to
$o.$ We investigate some geometric properties of $N$ by means of the functions
$k^{+},$ $k^{-}$ defined by (\ref{kk}).

\begin{lemma}
\label{hes1}Let $k\leq0$ be a continuous function such that $k^{+}(s)\leq
k(s)$ for $s\in\left[  0,+\infty\right)  .$ Let $F$ be the function defined as
in Lemma \ref{Fsao} with the actual $k.$ Then the inequality%
\[
\operatorname*{Hess}r\left(  x\right)  \left(  u,u\right)  \geq\frac
{F^{\prime}(r(x))}{F(r(x))}\left\langle u,u\right\rangle
\]
holds for all $x\in N\backslash\{o\}$ and all $u\in T_{x}N$ with
$u\bot\operatorname{grad}r(x).$
\end{lemma}

\begin{proof}
Let $S$ be the rotationally symmetric manifold with origin $o_{S}\in S$ such
that, in polar coordinates with origin $o_{S}$, the metric of $S$ is given by
$dt^{2}=ds^{2}+F(s)^{2}d\Theta^{2}.$ If $z\in S\backslash\{o_{S}\}$ and
$r_{S}(s)=d_{S}(z,o_{S})$ then
\[
K_{S}(z)(\operatorname{grad}r_{S},Z)=-\frac{F^{\prime\prime}(r_{S}%
(z))}{F(r_{S}(z))}=k(r_{S}(z)),
\]
for any $Z\in T_{z}S,$ $Z\bot\operatorname{grad}r_{S}$ (see \cite{Ch}). We
notice that since $k(s)\leq0,$ the exponential map of $S$ at $o_{S}$ is a
diffeomorphism and then $r_{S}$ is smooth on $S\backslash\{o_{S}\}$.

Let $x\in N\backslash\{o\}$ be given. The same proof as of the usual Hessian
Comparison Theorem (see Theorem 1.1 of \cite{SY}) gives%
\begin{align*}
\operatorname*{Hess}r\left(  x\right)  \left(  u,u\right)   &  \geq
\operatorname*{Hess}r_{S}\left(  z\right)  \left(  u_{S},u_{S}\right)
=\frac{F^{\prime}(r_{S}(z))}{F(r_{S}(z))}\left\langle u_{S},u_{S}\right\rangle
_{S}\\
&  =\frac{F^{\prime}(r(x))}{F(r(x))}\left\langle u,u\right\rangle ,
\end{align*}
where $z\in S,$ $r_{S}(z)=r(x),$ and $u_{S}$ is any vector in $T_{z}S$
orthogonal to $\operatorname{grad}r_{S}(z)$ such that $\left\langle
u,u\right\rangle =\left\langle u_{S},u_{S}\right\rangle _{S}.$
\end{proof}

\begin{corollary}
\label{hes2}Setting $G(s)=\int^{s}F(t)dt$ the inequality%
\[
\operatorname*{Hess}G\circ r\left(  x\right)  \left(  u,u\right)  \geq
F^{\prime}(r(x))\left\langle u,u\right\rangle
\]
holds for all $x\in N\backslash\{o\}$ and $u\in T_{x}N.$
\end{corollary}

\begin{proof}
One has
\[
\operatorname*{Hess}G\circ r\left(  u,u\right)  =G^{\prime}%
(r)\operatorname*{Hess}r\left(  u,u\right)  +G^{\prime\prime}\left\langle
\operatorname{grad}r,u\right\rangle ^{2},
\]
so that for $u\bot\operatorname{grad}r(x)$ we obtain from Lemma \ref{hes1}
\[
\operatorname*{Hess}G\circ r\left(  u,u\right)  =F(r)\operatorname*{Hess}%
r\left(  u,u\right)  \geq F^{\prime}(r)\left\langle u,u\right\rangle .
\]
Since $\operatorname*{Hess}r\left(  \operatorname{grad}r,\operatorname{grad}%
r\right)  =0$ we get
\[
\operatorname*{Hess}G\circ r\left(  \operatorname{grad}r,\operatorname{grad}%
r\right)  =G^{\prime\prime}(r)=F^{\prime}(r).
\]

\end{proof}

We now prove an extension of the classical monotonicity formula for minimal
surfaces \cite{An}. It will be obvious from the proof that a corresponding
result holds for higher dimensional minimal submanifolds. The proof is an
adaptation of the proof of Theorem 1 of \cite{An} and we refer the reader to
this paper for details.

\begin{proposition}
\label{mon}Let $M$ be a minimal surface in $N$ which has no boundary inside
some geodesic ball $B_{R}$ of $N$ centered at $o.$ Assume that $k^{+}(s)\leq
k(s)$ for $s\in\left[  0,+\infty\right)  $ where $k\leq0$ is a continuous
non-increasing function$.$ Let $F$ be the function determined by $k$ as in
Lemma \ref{Fsao}$.$ Then the function
\begin{equation}
r\mapsto\frac{\operatorname*{Area}(M\cap B_{r})}{\int_{0}^{r}F(t)dt},\text{
}0<r\leq R \label{mni}%
\end{equation}
is non-decreasing$.$
\end{proposition}

\begin{proof}
For any point $z\in M,$ let $e_{1},e_{2}$ be an orthonormal basis of $T_{z}M$
such that $e_{2}=\operatorname{grad}^{\top}r/\left\vert \operatorname{grad}%
^{\top}r\right\vert ,$ where $\operatorname{grad}^{\top}r$ is the orthogonal
projection of $\operatorname{grad}r$ on $T_{z}M.$ From the first variational
formula we have%
\[
\sum_{j=1}^{2}\int_{M}\left\langle \nabla_{e_{j}}E,e_{j}\right\rangle =0
\]
where $E$ is any smooth vector field with compact support on $M.$ Choosing $E$
of the form%
\[
E=f(r)\chi_{s}r\operatorname{grad}r,
\]
where $f$ is a smooth function satisfying $f^{\prime}\leq0,$ to be explicitly
given later, and $\chi_{s}$ is a smooth approximation to the characteristic
function of $\left[  0,s\right]  ,\ $we obtain
\[
\int f\chi_{s}\sum_{j=1}^{2}\left\langle \nabla_{e_{j}}r\operatorname{grad}%
r,e_{j}\right\rangle =-\int\left(  f^{\prime}\chi_{s}+f\chi_{s}^{\prime
}\right)  r\left\vert \operatorname{grad}^{\top}r\right\vert ^{2}%
\]
and then
\[
\int\left(  f\chi_{s}\sum_{j=1}^{2}\left\langle \nabla_{e_{j}}%
r\operatorname{grad}r,e_{j}\right\rangle +rf^{\prime}\chi_{s}\right)
\leq-\int rf\chi_{s}^{\prime}%
\]
Using Lemma \ref{hes1} we obtain%
\[
\sum_{j=1}^{2}\left\langle \nabla_{e_{j}}r\operatorname{grad}r,e_{j}%
\right\rangle \geq\left\langle e_{2},\operatorname{grad}r\right\rangle
^{2}+\left\langle e_{2},e_{2}^{\bot}\right\rangle ^{2}\frac{rF^{\prime}%
(r)}{F(r)}+\frac{rF^{\prime}(r)}{F(r)}%
\]
where $e_{2}^{\bot}$ is the unit vector along the projection of $e_{2}$ on the
tangent plane of the geodesic sphere centered at $o$. From (\ref{efe}) it
follows that%
\[
\left\langle e_{2},\operatorname{grad}r\right\rangle ^{2}+\left\langle
e_{2},e_{2}^{\bot}\right\rangle ^{2}\frac{rF^{\prime}(r)}{F(r)}\geq
\left\langle e_{2},\operatorname{grad}r\right\rangle ^{2}+\left\langle
e_{2},e_{2}^{\bot}\right\rangle ^{2}=1
\]
and then
\begin{equation}
\int\chi_{s}\left[  f\left(  1+\frac{rF^{\prime}}{F}\right)  +rf^{\prime
}\right]  \leq-\int rf\chi_{s}^{\prime}. \label{monn}%
\end{equation}

Choosing $f$ as a solution of the ODE
\[
f(r)\left(  1+\frac{rF^{\prime}(r)}{F(r)}\right)  +rf^{\prime}(r)=1
\]
with $f(0)=1/2$ we obtain%
\[
f(r)=\frac{\int_{0}^{r}F(t)dt}{rF(r)}.
\]
From Lemma \ref{Fsao} (ii) we have that $f^{\prime}\leq0.$ Then, setting
$v(r)=\operatorname*{Vol}\left(  B_{r}\right)  $\ and using (\ref{monn}) we
arrive at $v(r)\leq rf(r)v^{\prime}(r)$ from which we easily obtain (\ref{mni}).
\end{proof}

\bigskip

In what follows we want to compare the metric on the given manifold $N$ with
the metric of a rotationally symmetric complete background manifold $S_{0}$ of
non-positive sectional curvature $k_{0}$ given as a function of the distance
to the origin $o_{0}$ in $S_{0}.$ As in Lemma \ref{Fsao} (iii) we assume that
$k_{0}(s)<0$ is continuous and non-increasing and, as before, we denote by
$F_{0}$ the solution of $F_{0}^{\prime\prime}+k_{0}F_{0}=0,$ $F_{0}(0)=0,$
$F_{0}^{\prime}(0)=1.$

\begin{lemma}
\label{jac}Let $\gamma:\left[  0,+\infty\right)  \rightarrow N$ be a unit
speed geodesic, $\gamma(0)=o,$ and $J$ be a normal Jacobi field along
$\gamma,$ $J(0)=0,$ $\left\Vert J^{\prime}(0)\right\Vert =1.$

(i) If $k^{-}(s)\geq k(s)$ for some continuous function $k$ and if $F$ is a
solution of (\ref{F}) then it follows that $\left\Vert J(s)\right\Vert \leq
F(s),$ $0\leq s<+\infty.$ Likewise, if $k^{+}(s)\leq k(s)\leq0$ then one has
$\left\Vert J(s)\right\Vert \geq F(s).$

(ii) We suppose that $k^{+}(s)\leq k_{0}(s),$ $0\leq s<\infty$ and%
\begin{equation}
\frac{k^{-}-k_{0}}{\sqrt{-k_{0}}}\in L^{1}\left(  \left[  0,\infty\right)
\right)  . \label{k}%
\end{equation}
Then the estimate%
\[
F_{0}(s)\leq\left\Vert J(s)\right\Vert \leq e^{-C}F_{0}(s)
\]
holds with%
\[
C=\frac{^{\pi}}{2}\int_{0}^{\infty}\frac{k^{-}(s)-k_{0}(s)}{\sqrt{-k_{0}(s)}%
}ds.
\]

\end{lemma}

\begin{proof}
(i) It is an immediate consequence of Rauch's comparison theorem by comparing
$\left\Vert J(s)\right\Vert $ with the norm of a Jacobi field, satisfying the
same initial conditions as $J,$ in a rotationally symmetric manifold with
curvature $k(s)$.

(ii) The statement follows directly from (i) and Lemma \ref{Fsao} (iii).
\end{proof}

\begin{corollary}
\label{six}Let $\operatorname*{Exp}:T_{o}N\rightarrow N$ be the exponential
map at the base point $o$ and let the condition (\ref{k}) above be satisfied.
Then there is a constant $C$ such that%
\[
\left\Vert w\right\Vert \frac{F_{0}\left(  \left\Vert v\right\Vert \right)
}{\left\Vert v\right\Vert }\leq\left\Vert d\operatorname*{Exp}(v)w\right\Vert
\leq C\left\Vert w\right\Vert \frac{F_{0}\left(  \left\Vert v\right\Vert
\right)  }{\left\Vert v\right\Vert }%
\]
holds for $v,w\in T_{o}N$ with $v\bot w,$ $v\neq0.$
\end{corollary}

\begin{proof}
As is well known, $d\operatorname*{Exp}(v)w=J(1)$ where $J$ is the Jacobi
field along $t\mapsto\operatorname*{Exp}(tv)$ with initial condition $J(0)=0$
and $J^{\prime}(0)=w.$ Let $J_{1}$ be the Jacobi field along $t\mapsto
\operatorname*{Exp}(t\left\Vert v\right\Vert ^{-1}v)$ with $J_{1}(0)=0,$
$J_{1}^{\prime}(0)=\left\Vert w\right\Vert ^{-1}w.$ Then
\[
\widetilde{J}:=\frac{\left\Vert w\right\Vert }{\left\Vert v\right\Vert }%
J_{1}\left(  \left\Vert v\right\Vert t\right)
\]
is a Jacobi field along $t\mapsto\operatorname*{Exp}(tv)$ with $\widetilde
{J}(0)=0,$ $\widetilde{J}^{\prime}(0)=\left\Vert w\right\Vert J_{1}^{\prime
}(0)=w.$ Hence, $\widetilde{J}=J,$ $J(1)=\frac{\left\Vert w\right\Vert
}{\left\Vert v\right\Vert }J_{1}\left(  \left\Vert v\right\Vert t\right)  $
and the Corollary follows from Lemma \ref{jac} (ii).
\end{proof}

\bigskip

In the next lemma we obtain a special metric in the ball model for complete
rotationally symmetric metrics with sectional curvature bounded by above by a
negative constant:

\begin{lemma}
\label{BR}Let a complete rotationally symmetric metric of radial sectional
curvature $k\ $\ be given$,$ where $k$ is a continuous function of arclenght
$s\in\left[  0,\infty\right)  .$ We assume furthermore that $k(s)\leq-a^{2}$
for some constant $a>0$ and for all $s.$ Then there are coordinates defined in
the unit ball $B=\left\{  x\in\mathbb{R}^{n}\text{
$\vert$
}\left\vert x\right\vert <1\right\}  $ in which the metric takes the form%
\[
f^{\prime}(\left\vert x\right\vert )^{2}dx^{2}%
\]
where $dx^{2}$ stands for the Euclidean metric and $f\in C^{2}\left(  \left[
0,1\right)  \right)  \cap C^{3}\left(  \left(  0,1\right)  \right)  $ is the
inverse function of
\[
g(r):=e^{-\int_{r}^{+\infty}\frac{dt}{F(t)}},\text{ }0<r<\infty,
\]
and $F$ is the solution of (\ref{F}) with the given curvature $k.$ The
function $g$ is of class $C^{2}\left(  \left[  0,+\infty\right)  \right)  \cap
C^{3}\left(  \left(  0,+\infty\right)  \right)  $ with $g^{\prime}(r)>0$ for
$r\geq0$ and hence $f^{\prime}(t)>0$ for $t\in\left[  0,1\right)  .$
\end{lemma}

\begin{proof}
It follows from Lemma \ref{Fsao} (i) that $F$ grows at least exponentially so
that $\int_{r}^{+\infty}dt/F(t)$ is finite for each $r>0.$ Since $F\in
C^{2}\left(  \left[  0,+\infty\right)  \right)  ,$ $F(0)=0$ and $F^{\prime
}(0)=1$ we get%
\[
\frac{1}{F(t)}=\frac{1}{t}+b(t)
\]
for some function $b\in C^{0}\left(  \left[  0,+\infty\right)  \right)  $ from
which it follows that $g(r)\rightarrow0$ and $g(r)/r\rightarrow c$
($r\rightarrow0)$ for some $c>0.$ Hence%
\[
g^{\prime}(r)=\frac{g(r)}{F(r)}=\frac{g(r)}{r}\frac{r}{F(r)}\rightarrow
c,\text{ }r\rightarrow0
\]
and $g\in C^{1}\left(  \left[  0,+\infty\right)  \right)  $ and $g^{\prime
}(r)>0$ for $r\geq0.$ Since%
\[
g^{\prime\prime}=\frac{g^{\prime}F-gF^{\prime}}{F^{2}}=\frac{g(1-F^{\prime}%
)}{F^{2}}%
\]
we conclude that $\lim_{r\rightarrow0}g^{\prime\prime}(r)$ exists and thus
$g\in C^{2}\left(  \left[  0,+\infty\right)  \right)  .$

Introducing the coordinates $\left(  r,\theta\right)  \in\left[  0,1\right)
\times\mathbb{S}^{n-1}$ by $x=g(r)\theta$ we find by direct computation%
\[
f^{\prime}\left(  \left\vert x\right\vert \right)  ^{2}dx^{2}=dr^{2}%
+F(r)^{2}d\theta^{2}%
\]
and, by a well known formula, the radial sectional curvature of this metric
is\ $-F^{\prime\prime}/F=k.$
\end{proof}

\bigskip

In the sequel we construct a special ball model of $N.$ We denote by $S$ the
rotationally symmetric Hadamard manifold with origin $o_{S}$ with radial
sectional curvature $k_{0}$ given as a function of the distance to $o_{S}.$ We
require that $k_{0}(s)\leq-a^{2}$ for all $s$ with some $a>0.$ It follows from
Lemma \ref{BR} that $S$ is isometric to the open unit ball $B\subset
\mathbb{R}^{n}$ endowed with a metric of the form%
\begin{equation}
\left\langle u,v\right\rangle _{0}=f^{\prime}\left(  \left\vert x\right\vert
\right)  ^{2}\left(  u,v\right)  ,\text{ }u,v\in T_{x}B, \label{mb}%
\end{equation}
where $\left(  \text{ . }\right)  $ denotes the Euclidean scalar product and
$f$ is given by Lemma \ref{BR}.

\begin{proposition}
\label{BM}$N$ is isometric to an open unit ball $B\subset\mathbb{R}^{n}$ with
a metric of the form%
\begin{equation}
\left\langle u,v\right\rangle =f^{\prime}\left(  \left\vert x\right\vert
\right)  ^{2}\left\langle u,v\right\rangle _{b},\text{ }u,v\in T_{x}B,
\label{mbb}%
\end{equation}
where $f$ is the function given in Lemma \ref{BR} and $\left\langle
u,v\right\rangle _{b}$ is a Riemannian metric on $B$ which is uniformly
bounded from above and from below by the Euclidean metric.

The balls with center $0\in B$ in the metric (\ref{mbb}) above are at the same
time balls in the metric (\ref{mb}) of the same radius and the geodesics of
(\ref{mbb}) passing through $0\in B$ are straight line segments. In
particular, it follows that the asymptotic boundary $\partial_{\infty}B$ of
$B$ with respect to the metric (\ref{mbb}) is identified with the topological
boundary $\partial B$ of $B$ via the map that associates to each point of
$x\in\partial B$ the equivalence class of the geodesic ray from $0$ to $x.$
\end{proposition}

\begin{proof}
Let
\[
\operatorname*{Exp}:T_{o}N\rightarrow N,\text{ }\exp:T_{0}S\rightarrow S
\]
be the corresponding exponential maps and let us choose a linear isometry
$j:T_{0}S\rightarrow T_{o}N.$ We then define the diffeomorphism%
\[
\Phi:\operatorname*{Exp}\circ j\circ\exp^{-1}:S\rightarrow N.
\]
Since the geodesics of $S$ passing through $0\in B$ are straight line segments
and since the exponential maps map straight lines through the origin to
geodesics, it is clear that $\Phi$ maps the straight lines segments passing
through $0\in B$ to geodesics of $N$ passing through the base point $o\in N.$
The classical Gauss lemma says that the exponential map is an isometry in the
radial direction, i. e.,%
\begin{equation}%
\begin{array}
[c]{c}%
\left\langle d\operatorname*{Exp}(v)(v),d\operatorname*{Exp}%
(v)(w)\right\rangle =\left\langle v,w\right\rangle ,\text{ }v,w\in T_{o}N\\
\left\langle d\exp(v)(v),d\exp(v)(w)\right\rangle =\left\langle
v,w\right\rangle _{0},\text{ }v,w\in T_{0}S.
\end{array}
\label{exp}%
\end{equation}
This implies
\begin{equation}
\left\Vert d\Phi(x)\operatorname{grad}r_{0}(x)\right\Vert =1 \label{um}%
\end{equation}
and hence $\Phi$ maps balls with center $0\in S$ onto balls in $N$ centered at
$o$ with the same radius.

We now claim that there is a constant $C>0$ such that%
\begin{equation}
\left\Vert u\right\Vert _{0}\leq\left\Vert d\Phi(x)u\right\Vert \leq
C\left\Vert u\right\Vert _{0} \label{eq}%
\end{equation}
holds for $x\in S$ and $u\in T_{x}S.$ If this is shown then Proposition
\ref{BM} is proved since (\ref{eq}) can be rewritten as%
\begin{align*}
\left(  u,u\right)   &  =\frac{\left\langle u,u\right\rangle _{0}}{f^{\prime
}\left(  \left\Vert x\right\Vert \right)  ^{2}}\leq\frac{1}{f^{\prime}\left(
\left\Vert x\right\Vert \right)  ^{2}}\left\langle d\Phi(x)u,d\Phi
(x)u\right\rangle \\
&  \leq C^{2}\frac{\left\langle u,u\right\rangle _{0}}{f^{\prime}\left(
\left\Vert x\right\Vert \right)  ^{2}}=C^{2}\left(  u,u\right)  .
\end{align*}

The inequality (\ref{eq}) is already clear for $u:=\operatorname{grad}%
r_{0}(x)$ by (\ref{um}). Let then $u\in T_{x}S$ with $u\bot\operatorname{grad}%
r_{0}(x).$ It follows from (\ref{exp}) that%
\[
w:=d(\exp^{-1})(x)(u)\bot d(\exp^{-1})(x)\operatorname{grad}r_{0}%
(x)=\lambda\exp^{-1}(x)
\]
for some $\lambda$ and hence $jw\bot j\exp^{-1}(x)$ so that we obtain from
Corollary \ref{six} with $v:=\exp^{-1}(x)$%
\begin{align*}
\left\Vert w\right\Vert \frac{F_{0}\left(  \left\Vert v\right\Vert \right)
}{\left\Vert v\right\Vert }  &  =\left\Vert jw\right\Vert \frac{F_{0}\left(
\left\Vert jv\right\Vert \right)  }{\left\Vert jv\right\Vert }\leq\left\Vert
d\operatorname*{Exp}(jv)jw\right\Vert =\left\Vert d\Phi(x)u\right\Vert \\
&  \leq C\left\Vert jw\right\Vert \frac{F_{0}\left(  \left\Vert jv\right\Vert
\right)  }{\left\Vert jv\right\Vert }=C\left\Vert w\right\Vert _{0}\frac
{F_{0}\left(  \left\Vert v\right\Vert _{0}\right)  }{\left\Vert v\right\Vert
_{0}}.
\end{align*}
But for the rotationally symmetric metric on $S$ we have%
\[
\left\Vert w\right\Vert _{0}\frac{F_{0}\left(  \left\Vert v\right\Vert
_{0}\right)  }{\left\Vert v\right\Vert _{0}}=\left\Vert d\exp(v)w\right\Vert
=\left\Vert u\right\Vert _{0}.
\]

\end{proof}

\section{\label{expan}The expanding minimal discs}

\qquad We remind the reader of the definitions of the functions $k^{+}$ and
$k^{-}$ given in (\ref{kk}) and the assumptions of Theorem \ref{main1},
namely: There is a continuous, non-increasing, strictly negative function
$k_{0}$ such that $k^{+}(x)\leq k_{0}(s),$ $0\leq s<+\infty,${}and
\[
\frac{k^{-}-k_{0}}{\sqrt{-k_{0}}}\in L^{1}\left(  \left[  0,\infty\right)
\right)  .
\]
This will be assumed for the rest of the paper.

\bigskip

Let $B$ be the ball model of $N$ given by Proposition \ref{BM}. Given a
rectifiable Jordan curve $\Gamma\subset\partial_{\infty}B$ let $\Gamma_{1}$ be
the radial projections of $\Gamma$ onto the unit sphere (in the metric of $N)$
centered at $0\in B$ and let $\gamma:=\operatorname*{Exp}^{-1}\left(
\Gamma_{1}\right)  .$ We may assume that $\left\Vert \gamma^{\prime
}\right\Vert =1$ and define the family of Jordan curves $\Gamma_{R}\subset N,$
$1<R<+\infty,$ by $\Gamma_{R}=\operatorname*{Exp}(R\gamma).$ Morrey's
existence theorem $\left[  \cite{Mo1},\text{ }\cite{Mo2}\right]  $ guarantees,
for each $R,$ the existence of a minimizing disc $M_{R}$ with boundary
$\Gamma_{R}$ given by a harmonic, conformal, possibly branched immersion%
\begin{equation}
u_{R}:D\rightarrow N,\text{ }D=\left\{  z\in\mathbb{R}^{2}\text{ }\left\vert
z\right\vert <1\right\}  \label{mbi}%
\end{equation}
where $u_{R}\in C^{2}\left(  D\right)  \cap C^{0}\left(  \overline{D}\right)
$ and $u_{R}|\partial D$ parametrizes $\Gamma_{R}$ one-to-one. We estimate the
area of $M_{R}$ by comparison with the cone $c(s,t)=\operatorname*{Exp}%
(t\gamma(s)),$ $0\leq t\leq R,$ $0\leq s\leq L.$ By direct computation and
Corollary \ref{six}%
\[
\operatorname*{area}(c)=\int_{0}^{R}\int_{0}^{L}\left\Vert
d\operatorname*{Exp}(t\gamma(s))(\gamma^{\prime}(s)\right\Vert \left\Vert
d\operatorname*{Exp}(t\gamma(s))(\gamma(s)\right\Vert dsdt\leq CLG_{0}(R)
\]
with $G_{0}(R)=\int_{0}^{R}F_{0}(t)dt,$ so that%
\begin{equation}
\operatorname*{area}(M_{R})\leq CLG_{0}(R). \label{armr}%
\end{equation}
We now apply the monotonicity formula, Proposition \ref{mon} with
$k(s)=k_{0}(s)$ and obtain%
\begin{equation}
\operatorname*{area}(M_{R}\cap B_{s}(o))\leq CLG_{0}(s). \label{a}%
\end{equation}

Recalling our ball model for $N,$ which we use standardly from now on, we
translate (\ref{a}) into a growth condition for Euclidean balls $B_{r}%
^{e}(0)\subset B$ which have radius $f(r)$ (see Lemma \ref{BR}) in $S$ and as
well in $N$ (Proposition \ref{BM}):%
\begin{equation}
\operatorname*{area}(M_{R}\cap B_{r}^{e})\leq CLG_{0}(f(r)). \label{ae}%
\end{equation}
The next lemma makes the decisive step towards the convergence proof of the
family of surfaces $M_{R}.$

\begin{lemma}
\label{areas}The areas of the family $M_{R}$ with respect to the metric
$\left\langle \text{ , }\right\rangle _{b}$ (see Proposition \ref{BM}) stay
bounded independently of $R.$ Moreover, the Euclidean energies of the
corresponding mappings $u_{R}$ stay bounded as well. There is a radius
$\rho>0$ such that each of the surfaces $M_{R}$ intersects $B_{\rho}(o)\subset
N.$
\end{lemma}

\begin{proof}
By Proposition \ref{BM} the surface area elements $d\omega_{b}$ in the metric
$\left\langle \text{ , }\right\rangle _{b}$ and $d\omega$ in $N$ stand in the
relation
\[
d\omega=f^{\prime}\left(  \left\vert x\right\vert \right)  ^{2}d\omega_{b}.
\]
If therefore we define%
\[
\operatorname*{A}(r)=\operatorname*{area}(M_{R}\cap B_{r}^{e})
\]
in $N$ and
\[
\operatorname*{A}\nolimits_{b}(r)=\operatorname*{area}(M_{R}\cap B_{r}^{e})
\]
in the $\left\langle \text{ , }\right\rangle _{b}$- metric, it follows from
the coarea formula $\left[  \cite{Fe},\text{ 3.2.12}\right]  $ that
\[
\frac{d}{dr}\operatorname*{A}\nolimits_{b}(r)=f^{\prime}(r)^{-2}\frac{d}%
{dr}\operatorname*{A}(r)
\]
from what we get by integration%
\begin{align}
\operatorname*{A}\nolimits_{b}(f^{-1}(R))-\operatorname*{A}\nolimits_{b}%
(f^{-1}(\rho))  &  =\int_{f^{-1}(\rho)}^{f^{-1}(R)}f^{\prime}(r)^{-2}%
\operatorname*{A}\nolimits^{\prime}(r)dr\nonumber\\
&  =f^{\prime}(f^{-1}(R))^{-2}\operatorname*{A}(f^{-1}(R))-f^{\prime}%
(f^{-1}(\rho))^{-2}\operatorname*{A}(f^{-1}(\rho))\nonumber\\
&  +2\int_{f^{-1}(\rho)}^{f^{-1}(R)}\frac{f^{\prime\prime}(r)}{f^{\prime
}(r)^{3}}\operatorname*{A}(r)dr. \label{set}%
\end{align}
From Lemma \ref{BR} we recall the following relations%
\begin{align*}
f^{\prime}  &  =\frac{1}{g^{\prime}\circ f}=\frac{F_{0}\circ f}{g\circ f}\\
f^{\prime\prime}  &  =\frac{F_{0}^{\prime}g-F_{0}g^{\prime}}{g^{2}}\circ
ff^{\prime}=\frac{\left(  F_{0}^{\prime}-1\right)  F_{0}}{g^{2}}\circ f.
\end{align*}
Since $F_{0}^{\prime}(0)=1$ and $F_{0}^{\prime\prime}\geq0$ we see that
\begin{equation}
f^{\prime\prime}\geq0. \label{f2}%
\end{equation}
Lemma \ref{Fsao} (i) implies $F_{0}^{\prime}\geq cF_{0}$ for some constant
$c>0$ from what we get%
\begin{equation}
G_{0}\leq c^{-1}F_{0}. \label{goc}%
\end{equation}
By means of (\ref{ae}), (\ref{f2}) and (\ref{goc}) we may continue the
estimate (\ref{set}):%
\begin{align}
\operatorname*{A}\nolimits_{b}(f^{-1}(R))-\operatorname*{A}\nolimits_{b}%
(f^{-1}(\rho))  &  \leq CL\left(  \frac{g(R)^{2}}{F_{0}(R)^{2}}G_{0}(R)\right.
\nonumber\\
&  \left.  +2\int_{f^{-1}(\rho)}^{f^{-1}(R)}\frac{f^{\prime\prime}%
(r)}{f^{\prime}(r)^{2}}\frac{g(f(r))G_{0}(f(r))}{F_{0}(f(r))}dr\right)
\nonumber\\
&  \leq CLc^{-1}\left(  \frac{g(R)}{F_{0}(R)}+2\int_{f^{-1}(\rho)}^{f^{-1}%
(R)}\left(  \frac{-1}{f^{\prime}(r)}\right)  ^{\prime}dr\right) \nonumber\\
&  \leq2CLc^{-1}\frac{g(\rho)}{F_{0}(\rho)} \label{ten}%
\end{align}
for arbitrary $\rho\in\left(  0,R\right)  .$ Because of $0\leq g\leq1,$
$g(0)=0$ and $g^{\prime}(0)>0$ as we showed in Lemma \ref{BR}, we see that the
areas of the surfaces $M_{R}$ in the $\left\langle \text{ , }\right\rangle
_{b}-$metric stay bounded independently of $R.$ The maps $u_{R}:D\rightarrow
N$ being conformal in the metric of $N$ are conformal with respect to the
$\left\langle \text{ , }\right\rangle _{b}-$metric as well since this metric
differs from the one of $N$ by a conformal factor. But then it follows that
the energies of the mappings $u_{R}$ in the $\left\langle \text{ ,
}\right\rangle _{b}-$metric and, on account of Proposition \ref{BM}, as well
in the Euclidean metric are bounded independently of $R.$ In other words the
mapping $u_{R},$ considered as mappings into $\mathbb{R}^{n}$ are bounded in
the norm of the Sobolev space $H_{2}^{1}.$

Let us now assume that $M_{R}$ omits the ball $B_{\rho}(o)$ of $N.$ Then one
sees from (\ref{ten}) that the area of $M_{R}$ in the $\left\langle \text{ ,
}\right\rangle _{b}-$metric and hence the Euclidean energy of $u_{R}$ become
arbitrarily small if $R$ and $\rho$ are sufficiently large. This however
contradicts the fact that $u_{R}|_{\partial D}$ parametrizes a rectifiable
Jordan curve $\Gamma_{R}\subset B$ and $\Gamma_{R}$ converges to a rectifiable
curve $\Gamma\subset\partial_{\infty}B$ as $R\rightarrow\infty.$ Here we used
the fact that in the ball model of $N$ geodesic cones with center $0$ are
straight Euclidean cones. This shows that there is a ball $B_{\rho}(o)\subset
N$ such that $B_{\rho}(o)\cap M_{R}\neq\varnothing$ for all $R\geq1.$
\end{proof}

\bigskip

In the next lemma we prove local energy and local $C^{0}-$estimates for
conformal harmonic maps $u:D\rightarrow N.$

\begin{lemma}
\label{c0}Let $u:D\rightarrow N$ be harmonic and conformal.

(i) For any subset $D_{0}$ with $\overline{D}_{0}\subset D$ holds%
\[
E(u,D_{0}):=\frac{1}{2}\int_{D_{0}}\left\Vert du\right\Vert ^{2}dx\leq
8a^{-2}\operatorname*{cap}(D_{0},D)
\]
with%
\[
\operatorname*{cap}(D_{0},D)=\inf\left\{  \frac{1}{2}\int_{D_{0}}\left\vert
d\eta\right\vert ^{2}dx\text{
$\vert$
}\eta\in C_{0}^{\infty}\left(  D\right)  ,\text{ }\eta=1\text{ on }%
D_{0}\right\}
\]

(ii) For any $z_{0}\in D$ and $s<\left(  1-\left\vert z_{0}\right\vert
\right)  ^{2}$ we have the estimate%
\[
\operatorname*{dist}\left(  u(z),u(z_{0})\right)  \leq\left(  \sqrt{\frac
{8}{\pi}}+4\sqrt{\frac{\pi}{-\ln s}}\right)  a^{-2}\operatorname*{cap}%
(D_{\sqrt{s}}(z_{0}),D)^{\frac{1}{2}}%
\]
for $z\in D_{s}(z_{0}):=\left\{  z\in D\text{
$\vert$
}\left\vert z-z_{0}\right\vert <s\right\}  .$
\end{lemma}

\begin{proof}
(i) We set $k=k^{+}$ in Lemma \ref{Fsao} and Lemma \ref{hes1} and consider the
function%
\[
w:=G\circ r\circ u,\text{ }r(x)=\operatorname*{dist}(x,o).
\]
Using the harmonicity of $u$ we obtain from the Corollary \ref{hes2}%
\[
\Delta w=\sum_{k=1,2}\left(  \operatorname*{Hess}G\circ r\right)  (u)\left(
\frac{\partial u}{\partial x_{k}},\frac{\partial u}{\partial x_{k}}\right)
\geq F^{\prime}(r(u))\left\Vert du\right\Vert ^{2}.
\]
We test this inequality with the function $\varphi=\eta^{2}/F\circ r\circ u$
where $\eta\in C_{0}^{\infty}\left(  D\right)  ,$ $\eta=1$ on $D_{0}$ to
obtain%
\begin{align*}
0  &  \geq\int\left(  \sum_{k=1,2}\frac{\partial w}{\partial x_{k}}%
\frac{\partial\varphi}{\partial x_{k}}+F^{\prime}\left\Vert du\right\Vert
^{2}\varphi\right)  dx\\
&  =\int\left(  -\eta^{2}\frac{G^{\prime}F^{\prime}}{F^{2}}\sum_{k=1,2}%
\left\langle \operatorname{grad}r,\frac{\partial u}{\partial x_{k}%
}\right\rangle ^{2}+\eta^{2}\frac{F^{\prime}}{F}\left\Vert du\right\Vert
^{2}\right. \\
&  \left.  +2\eta\frac{G^{\prime}}{F}\sum_{k=1,2}\frac{\partial\eta}{\partial
x_{k}}\left\langle \operatorname{grad}r,\frac{\partial u}{\partial x_{k}%
}\right\rangle \right)  dx.
\end{align*}
Since $G^{\prime}=F$ this simplifies to%
\begin{equation}%
\begin{array}
[c]{c}%
\int\left(  \eta^{2}\frac{F^{\prime}}{F}\sum_{k=1,2}\left(  \left\Vert
\frac{\partial u}{\partial x_{k}}\right\Vert ^{2}-\left\langle
\operatorname{grad}r,\frac{\partial u}{\partial x_{k}}\right\rangle
^{2}\right)  dx\right. \\
\leq2\int\eta\left\vert d\eta\right\vert \left(  \sum_{k=1,2}\left\langle
\operatorname{grad}r,\frac{\partial u}{\partial x_{k}}\right\rangle
^{2}\right)  ^{\frac{1}{2}}dx.
\end{array}
\label{simp}%
\end{equation}
Let now $x\in D$ be arbitrary and $\left(  e_{1},...,e_{n}\right)  $ be an
orthonormal basis at $u(x)$ with $e_{1}=\operatorname{grad}r.$ The
conformality relations then read%
\begin{align*}
0  &  =\left\langle \frac{\partial u}{\partial x_{1}},\frac{\partial
u}{\partial x_{2}}\right\rangle =\sum_{j=1}^{n}\left\langle e_{j}%
,\frac{\partial u}{\partial x_{1}}\right\rangle \left\langle e_{j}%
,\frac{\partial u}{\partial x_{2}}\right\rangle ,\\
0  &  =\left\Vert \frac{\partial u}{\partial x_{1}}\right\Vert ^{2}-\left\Vert
\frac{\partial u}{\partial x_{2}}\right\Vert ^{2}=\sum_{j=1}^{n}\left(
\left\langle e_{j},\frac{\partial u}{\partial x_{1}}\right\rangle
^{2}-\left\langle e_{j},\frac{\partial u}{\partial x_{2}}\right\rangle
^{2}\right)
\end{align*}
which, in complex notation with $i=\sqrt{-1},$ may be rewritten in the form%
\[
\sum_{j=1}^{n}\left(  \left\langle e_{j},\frac{\partial u}{\partial x_{1}%
}\right\rangle +i\left\langle e_{j},\frac{\partial u}{\partial x_{2}%
}\right\rangle \right)  ^{2}=0.
\]
Separating the term $j=1$ in this sum allows to estimate the radial component
$du^{rad}$ of $du$ by the spherical component $du^{spher},$ $i.e.$%
\begin{equation}
\left\Vert du^{rad}\right\Vert ^{2}\leq\left\Vert du^{spher}\right\Vert ^{2}.
\label{d}%
\end{equation}
We now employ Lemma \ref{Fsao} (i) and (\ref{d}) to obtain from (\ref{simp})%
\begin{align*}
&  a\int\eta^{2}\left\Vert du^{spher}\right\Vert ^{2}dx\\
&  \leq\left(  \varepsilon\int\eta^{2}\left\Vert du^{spher}\right\Vert
^{2}dx+\frac{1}{\varepsilon}\int\left\vert d\eta\right\vert ^{2}dx\right)
\end{align*}
for arbitrary $\varepsilon>0.$ Choosing $\varepsilon=\frac{1}{2}a$ yields%
\[
\int\eta^{2}\left\Vert du^{spher}\right\Vert ^{2}dx\leq4a^{-2}\int\left\vert
d\eta\right\vert ^{2}dx.
\]
Using (\ref{d}) once more proves the statement.

(ii) For any $r<1-\left\vert z_{0}\right\vert $ the image of $u|D_{r}(z_{0})$
contains a minimal surface which passes through $u(z_{0})$ and has no boundary
inside the geodesic ball centered at $u(z_{0})$ and of radius%
\[
\delta(r):=\inf\left\{  d(u(z),u(z_{0}))\text{
$\vert$
}z\in\partial D_{r}(z_{0})\right\}  .
\]
The classical monotonicity formula for non positively curved metrics gives the
estimate%
\[
\operatorname*{area}\left(  u|D_{r}(z_{0})\right)  \geq\pi\delta(r)^{2}.
\]
Part (i) then provides%
\begin{equation}
\delta(r)\leq\sqrt{\frac{8}{\pi}}a^{-1}\operatorname*{cap}(D_{r}%
(z_{0}),D)^{\frac{1}{2}}. \label{del}%
\end{equation}

Given now $s\in\left(  0,(1-\left\vert z_{0}\right\vert )^{2}\right)  $ the
Lemma of Courant-Lebesgue [\cite{DHS}, 4.4] guarantees the existence of a
radius $r\in\left(  s,\sqrt{s}\right)  $ such that the length of the curve
$u|\partial D_{r}(z_{0})$ is estimated as follows:%
\begin{align*}
L\left(  u|\partial D_{r}(z_{0})\right)   &  \leq\sqrt{\frac{8\pi}{-\ln s}%
}E\left(  u|D_{\sqrt{s}}(z_{0})\right)  ^{\frac{1}{2}}\\
&  \leq8\sqrt{\frac{\pi}{-\ln s}}a^{-1}\operatorname*{cap}(D_{\sqrt{s}}%
(z_{0}),D)^{\frac{1}{2}}.
\end{align*}
Combining this with (\ref{del}) we arrive at%
\begin{align*}
&  \sup\left\{  \operatorname*{dist}(u(z),u(z_{0})\text{
$\vert$
}z\in\partial D_{r}(z_{0})\right\} \\
&  \leq\left(  \sqrt{\frac{8}{\pi}}+4\sqrt{\frac{\pi}{-\ln s}}\right)
a^{-1}\operatorname*{cap}(D_{\sqrt{s}}(z_{0}),D)^{\frac{1}{2}}.
\end{align*}
The maximum principle for harmonic maps into non-positively curved spaces
\cite{JK} yields the statement in (ii).
\end{proof}

\bigskip

After a suitable conformal reparametrization of $u_{R}:D\rightarrow N,$ we may
assume that the following important normalization holds%
\begin{equation}
u_{R}(0)\in B_{\rho}(o)\subset N,\text{ }0<R<+\infty, \label{treze}%
\end{equation}
where $\rho$ is given by Lemma \ref{areas}.

We are now in position to prove:

\begin{lemma}
\label{main}

(i) For some sequence $R_{k}\rightarrow+\infty$ the sequence of conformal
harmonic maps $\left(  u_{R_{k}}\right)  $ converges locally in $C^{2}$ to a
proper, conformal harmonic map $u:D\rightarrow N.$ If $n=3$ then $u$ is an embedding.

(ii) Considered as a map into $\mathbb{R}^{n},$ $u$ belongs to the Sobolev
space $H_{2}^{1}\left(  D,\mathbb{R}^{n}\right)  $ and its trace $u|\partial
D$ either is a continuous weakly monotonic parametrization of $\Gamma
=\lim_{R\rightarrow\infty}\Gamma_{R}\subset\partial B$ or $u|\partial D$
equals one point of $\Gamma.$
\end{lemma}

\begin{proof}
We shall obtain $u$ as a limit of a subsequence of the sequence $u_{R}%
:D\rightarrow N$ given by (\ref{mbi}). Due the normalization condition
(\ref{treze}) and by Lemma \ref{c0}, for each subdisc $D_{r}\left(  0\right)
\subset D$ with $r<1$ all the maps $u_{R}$ map $D_{r}(0)$ into some fixed ball
$B_{s(r)}(o)\subset N$ and the energies of $u_{R}|D_{r}(0)$ are uniformly
bounded as well. This makes Morrey's H\"{o}lder estimate for energy minimizing
maps applicable \cite{Mo2} so that we get an uniform $C^{\alpha}-$bound for
$u_{R}$ on each subdisc $D_{r}(0)$ for some $\alpha(r)\in\left(  0,1\right)
.$ By well known regularity estimates for harmonic maps this implies uniform
local $C^{2,\alpha}$ bounds for the family $u_{R}.$ Therefore we may find a
sequence $R_{k}\rightarrow\infty$ such that $u_{R_{k}}$ converges locally in
$C^{2}$ to a conformal harmonic map $u:D\rightarrow N.$ On the other hand,
considering the $u_{R}$ as maps into $\mathbb{R}^{n},$ we know from Lemma
\ref{areas} that the $u_{R}$ are uniformly bounded in the $H_{2}%
^{1}(D,\mathbb{R}^{n})-$norm, so that we may also assume that $u_{R_{k}}$
converges to $u$ weakly in $H_{2}^{1}(D,\mathbb{R}^{n}).$ The trace operator
$H_{2}^{1}(D,\mathbb{R}^{n})\rightarrow L_{2}(\partial D,\mathbb{R}^{n})$
being compact we may then furthermore assume that $u_{R_{k}}|\partial
D\rightarrow u|\partial D$ ($k\rightarrow+\infty)$ in $L_{2}(\partial
D,\mathbb{R}^{n}).$

Let us now choose parametrizations $\gamma_{R}:\left[  0,2\pi\right]
\rightarrow\Gamma_{R},$ $\gamma:\left[  0,2\pi\right]  \rightarrow\Gamma$
which are proportional to Euclidean arclength such that $\gamma_{R}%
\rightarrow\gamma$ uniformly in the Euclidean metric as $R\rightarrow\infty.$
We extend $\gamma_{R}$ and $\gamma$ as periodic functions defined on
$\mathbb{R}$. Then we may write%
\begin{equation}
u_{R}(e^{i\theta})=\gamma_{R}(\varphi_{R}(\theta)),\text{ }0\leq\theta\leq
2\pi, \label{ei}%
\end{equation}
with some monotonic function $\varphi_{R},$ $0\leq\varphi_{R}(0)\leq2\pi,$
$\varphi_{R}(2\pi)-\varphi_{R}(0)=2\pi.$

A classical theorem of Helly says that any sequence of monotone, uniformly
bounded functions has a pointwise convergent subsequence, so that we may also
assume that
\[
u_{R_{k}}(e^{i\theta})\rightarrow\gamma(\varphi(\theta))\text{ (}%
k\rightarrow+\infty)
\]
for some monotone function $\varphi$ with $\varphi(2\pi)-\varphi(0)=2\pi.$
Together with the $L_{2}$ convergence $u_{R_{k}}|\partial D\rightarrow
u|\partial D$ this clearly implies that
\[
u(e^{i\theta})=\gamma(\varphi(\theta)),\text{ }0\leq\theta\leq2\pi.
\]
This shows that $u|\partial D$ could only have jump discontinuities; however
these are not possible for boundary values of an $H_{2}^{1}-$function thanks
to the lemma of Courant-Lebesgue [\cite{DHS}, 4.4]. It follows that $\varphi$
cannot have jumps of height less than $2\pi$ and we arrive at the alternative
that either $\varphi$ is continuous or it makes a jump of $2\pi,$ in other
words, either $u(\partial D)=\Gamma$ or $u|\partial D$ is constant.

As next, let us show that the limit map $u$ is proper. From what we already
showed above we know that $u(\partial D)\subset\partial_{\infty}B$ which, in
the metric of $N,$ means that the Sobolev trace $u(\partial D)$ is infinitely
far away. Let a ball $B_{R}(o)$ be given with arbitrarily large $R.$ Since
$\overline{B_{2R}(o)}\subset N$ is a compact subset of $B$ and $u(\partial
D)\subset\partial_{\infty}B$ we can apply Theorem 1 in \cite{DHT} to find a
sequence of radii $r_{k}\rightarrow1$ ($k\rightarrow+\infty),$ $r_{k}%
<r_{k+1}<1,$ such that $u\left(  \partial D_{r_{k}}(0)\right)  \subset
B\backslash B_{2R}(o).$ Let us set%
\[
A_{k}:=\left\{  z\in D\text{
$\vert$
}r_{k}<\left\vert z\right\vert <r_{k+1}\right\}
\]
and let us assume that some $A_{k}$ contains points $z_{k}$ with $u(z_{k})\in
B_{R}(o)\subset N.$ Since $u(\partial A_{k})$ is outside of $B_{2R}(o),$
$u(A_{k})\cap B_{2R}(o)$ contains a minimal surface which passes through
$u\left(  z_{k}\right)  \in B_{R}(o)$ and has no boundary inside $B_{2R}(o)$
so that the monotonicity formula gives
\begin{equation}
\operatorname*{area}\left(  u(A_{k})\cap B_{2R}(o)\right)  \geq\pi R^{2}.
\label{ak}%
\end{equation}
Since $u\in H_{2}^{1}\left(  D,\mathbb{R}^{n}\right)  $ the area of $u$ inside
$B_{2R}(o)\subset N$ is finite so that (\ref{ak}) can hold only for finitely
many $k$ and hence $d(u(z),o)\geq R$ for $\left\vert z\right\vert \geq
r_{k_{0}}$ for some $k_{0},$ showing that $u:D\rightarrow N$ is proper.

Let us finally consider the case $n=3.$ Since the the boundary curves of the
surfaces $u_{R}$ are contained in the metric spheres of $N$ and the spheres
are convex, it follows from the results in \cite{G}, \cite{MY} that $u_{R}$ is
an embedding. Then, as a limit of minimal embeddings, $u$ is an embedding,
too. This concludes the proof of the Lemma.
\end{proof}

\section{\label{Bu}The blowing up procedure and proof of the Theorem
\ref{main1}}

\qquad The concentration phenomenon which comes up as a possibility in the
limiting process in Lemma \ref{main} and the resulting splitting off of a
punctured minimal sphere can be excluded if one can construct suitable
foliations of the space by convex hypersurfaces. Such foliations are obvious
in the hyperbolic space but do exist also in more general Hadamard manifolds,
as explained in the next section. Instead we shall now set up a blow up
procedure, magnifying neighborhoods of the point where the concentration
happens. The splitting off of punctured minimal spheres may repeat itself,
however we can show that after a finitely many split offs a solution to the
asymptotic Plateau problem remains.

\begin{proof}
[Proof of Theorem \ref{main1}]In the proof of Lemma \ref{areas} it was already
used that there is a positive lower bound for the euclidean area of discs
spanned by one of the curves $\Gamma_{R},$ $R\geq1.$ We need a corresponding
statement for a family of curves which are obtained from the $\Gamma_{R}$ by
the following modifications: One takes out a subarc $\alpha$ from $\Gamma_{R}$
of Euclidean length not exceeding $\varepsilon>0$ and replaces it by some
other rectifiable arc $\beta$ of length at most $\delta.$ If $\widetilde{M}$
is a disc spanned by such a modified curve $\widetilde{\Gamma}_{R}$ one may
produce a disc $M$ filling the original $\Gamma_{R}$ by attaching a cone over
$\alpha\cup\beta$ along the boundary segment $\beta$ of $\widetilde{M}$ and
hence%
\[
\operatorname*{area}\nolimits^{e}\left(  M\right)  \leq\operatorname*{area}%
\nolimits^{e}\left(  \widetilde{M}\right)  +\left(  \varepsilon+\delta\right)
^{2}.
\]
If therefore $\varepsilon$ and $\delta$ are sufficiently small we see that
there is $a_{0}>0$ such that
\begin{equation}
\operatorname*{area}\nolimits^{e}\left(  \widetilde{M}\right)  \geq a_{0},
\label{are}%
\end{equation}
where $\operatorname*{area}\nolimits^{e}$ denotes the euclidean area, for all
discs spanned by some $\widetilde{\Gamma}_{R},$ $R\geq1.$ Let us now return to
the representation (\ref{ei}) for the boundary data of the family
$u_{R\text{:}}$%
\[
u_{R}\left(  e^{i\theta}\right)  =\gamma_{R}\left(  \varphi_{R}(\theta
)\right)  ,\text{ }0\leq\theta\leq2\pi,
\]
$\gamma_{R}$ being a proportional-to-arclength parametrization.

If for a sequence $R_{k}\rightarrow+\infty$ the sequence $\left(
\varphi_{R_{k}}\right)  $ converges pointwise to a step function with one jump
of height $2\pi$ we may (after a rotation of $D)$ assume that the jump occurs
at $\theta=\pi.$ After passing to a subsequence we may assume that
\[
\varphi_{R_{k}}\left(  \pi+\frac{1}{k}\right)  -\varphi_{R_{k}}\left(
\pi-\frac{1}{k}\right)  >2\pi-\frac{1}{k}%
\]
so that $\gamma_{R_{k}}\circ\varphi_{R_{k}}\left(  \left[  \pi-1/k,\pi
+1/k\right]  \right)  $ represents a subarc of euclidean length at least
$\left(  1-1\backslash(2k\pi)\right)  \operatorname*{length}\left(
\Gamma_{R_{k}}\right)  $ and the complementary subarc of $\Gamma_{R_{k}}$ has
length at most $\left(  1/(2k\pi)\right)  \operatorname*{length}\left(
\Gamma_{R_{k}}\right)  .$ The lemma of Courant-Lebesgue [\cite{DHS}, 4.4]
provides a radius $r_{k}\in\left(  1/k,1/\sqrt{k}\right)  $ such that
\[
\operatorname*{length}\left(  u_{R_{k}}\left(  D\cap\partial D_{r_{k}%
}(-1)\right)  \right)  \leq\sqrt{\frac{8\pi E_{0}}{\ln k}}%
\]
where $E_{0}$ is an upper bound for the euclidean energies of $u_{R_{k}}.$ For
sufficiently large $k$ the curve $\widetilde{\Gamma}_{R_{k}}:=u_{R_{k}}\left(
\partial\left(  D\cap D_{r_{k}}(-1)\right)  \right)  $ satisfies the
conditions required for inequality (\ref{are}), making it obvious that a
concentration of energy takes place near the boundary point $-1.$ Since
$u_{R_{k}}|D\cap D_{r_{k}}(-1)$ is part of the surface $u_{R_{k}},$ (\ref{a})
and the estimates of Lemma \ref{areas} trivially remain valid for $u_{R_{k}%
}|D\cap D_{r_{k}}(-1),$ irrespective of the modification of the boundary
curve. But then (\ref{ten}) also holds with $R=R_{k}$ showing that there is a
radius $\widetilde{\rho}>0$ only depending on $\Gamma,$ $a_{0}$ and the
geometry of $N$ such that%
\[
u_{R_{k}}\left(  D\cap D_{r_{k}}(-1)\right)  \cap B_{\widetilde{\rho}}\left(
0\right)  \neq\emptyset
\]
for all sufficiently large $k,$ unless (\ref{are}) were violated. Therefore we
may now choose conformal maps $T_{k}:D\rightarrow D\cap D_{r_{r}}(-1)$ such
that $\widetilde{u}_{R_{k}}:=u_{R_{k}}\circ T_{k}$ satisfies%
\begin{equation}
\widetilde{u}_{R_{k}}\left(  0\right)  \in B_{\widetilde{\rho}}\left(
0\right)  \subset N. \label{point}%
\end{equation}

Let us now look at the minimal surface $u_{R_{k}}|D\backslash D_{r_{k}}(-1)$
which tends to a punctured sphere for $k\rightarrow\infty.$ Recalling the
condition $u_{R_{k}}\left(  0\right)  \in B_{\rho}\left(  0\right)  \subset N$
and observing that
\[
u_{R_{k}}\left(  \partial\left(  D\backslash D_{r_{k}}(-1)\right)  \right)
\cap B_{2\rho}\left(  0\right)  =\emptyset
\]
for sufficiently large $k$ we obtain from the monotonicity formula \ref{mon}
\[
\operatorname*{area}\nolimits^{N}\left(  u_{R_{k}}\left(  \left(  D\backslash
D_{r_{k}}(-1)\right)  \right)  \cap B_{2\rho}\left(  0\right)  \right)
\geq\pi\rho^{2},
\]
which in view of Proposition \ref{BM} leads to an estimate of the euclidean
energy of $u_{R_{k}}|D\cap D_{r_{k}}(-1)$ of the form%
\begin{equation}
E\left(  u_{R_{k}}\left(  \left(  D\backslash D_{r_{k}}(-1)\right)  \right)
\right)  \geq\operatorname*{e}\left(  \rho\right)  , \label{dset}%
\end{equation}
where $\operatorname*{e}\left(  \rho\right)  $ depends only on $\rho$ and the
geometry of $N.$ Recalling that $E_{0}$ was an upper bound for the euclidean
energies of the sequence $u_{R_{k}}$ we thus see that
\begin{equation}
E\left(  \widetilde{u}_{R_{k}}\right)  \leq E_{0}-\operatorname*{e}\left(
\rho\right)  , \label{des8}%
\end{equation}
i.e. the splitting off of a punctured minimal sphere reduces the energy by a
fixed amount. We may now apply the same analysis as in the proof of Lemma
\ref{main} to the sequence $\left(  \widetilde{u}_{R_{k}}\right)  $ resulting
in the convergence locally in $C^{2}$ and weakly in $H_{2}^{1}\left(
D,\mathbb{R}^{n}\right)  $ of a subsequence of $\left(  \widetilde{u}_{R_{k}%
}\right)  $ towards a conformal, harmonic, proper map from $D$ to $N.$ Let us
now investigate the behavior of the boundary values of $\widetilde{u}_{R_{k}%
}|D$ which parametrize the curve $\widetilde{\Gamma}_{R_{k}}$ monotonically.

We recall that $\widetilde{\Gamma}_{R_{k}}$ consists of a subarc $\alpha_{k}$
of $\Gamma_{R_{k}}$ of length at least $\left(  1-1/(2k\pi)\right)
\operatorname*{length}\left(  \Gamma_{R_{k}}\right)  $ and with endpoints
$u_{R_{k}}\left(  \partial D\cap\partial D_{r_{k}}(-1)\right)  $ together with
the arc $\beta_{k}=u_{R_{k}}\left(  D\cap\partial D_{r_{k}}(-1)\right)  $ of
length at most $\sqrt{8\pi E_{0}/\ln k}.$ We choose proportional-to-arclength
parametrizations $\widetilde{\gamma}_{k}:\left[  0,2\pi\right]  \rightarrow
\widetilde{\Gamma}_{R_{k}}$ such that $\widetilde{\gamma}_{k}(0)$ is an
endpoint of $\alpha_{k}.$ Passing to a subsequence we have $\widetilde{\gamma
}_{k}\rightarrow\gamma$ uniformly, where $\gamma$ is a
proportional-to-arclength parametrization of $\Gamma.$ The representation
$\widetilde{u}_{R_{k}}\left(  e^{i\theta}\right)  =\widetilde{\gamma}%
_{k}\left(  \varphi_{k}(\theta)\right)  $ holds with monotone functions
$\varphi_{k},$ $\varphi_{k}(2\pi)-\varphi_{k}(0)=2\pi.$ After a rotation of
$D$ we may assume that $\varphi_{k}(0)=0$ and hence $\varphi_{k}(2\pi)=2\pi.$
Let us choose $\theta_{k}\in\left(  0,2\pi\right)  $ such that $\widetilde
{\gamma}_{k}\circ\varphi_{k}|\left[  0,\theta_{k}\right]  $ parametrizes
$\alpha_{k}$ and $\widetilde{\gamma}_{k}\circ\varphi_{k}|\left[  \theta
_{k},2\pi\right]  $ parametrizes $\beta_{k}.$ Then clearly%
\begin{equation}
\varphi_{k}(\theta_{k})\rightarrow2\pi\text{ }\left(  k\rightarrow
\infty\right)  . \label{dis9}%
\end{equation}
After passing to a subsequence we may assume that $\theta_{k}\rightarrow
\widetilde{\theta}\in\left[  0,2\pi\right]  $ $\left(  k\rightarrow
\infty\right)  ,$ $\varphi_{k}\rightarrow\varphi$ pointwise on $\left[
0,2\pi\right]  $ and $\widetilde{u}_{R_{k}}\rightarrow\widetilde{u}$ locally
in $C^{2}$ and weakly in $H_{2}^{1}\left(  D,\mathbb{R}^{n}\right)  $ where
$\widetilde{u}:D\rightarrow N$ is a harmonic, conformal, proper map.

If $\widetilde{\theta}=0$ then $\varphi(\theta)=2\pi$ on $\left(
0,2\pi\right]  $ and hence $\widetilde{u}|_{\partial D}=\gamma(2\pi),$ i.e. a
punctured minimal sphere has split off. Let us consider the case that
$\widetilde{\theta}>0.$

It follows from (\ref{dis9}) that $\varphi(\theta)=2\pi$ for all $\theta
\in\left(  \widetilde{\theta},2\pi\right]  $ and $\varphi\left(
\widetilde{\theta}-0\right)  =2\pi$ so that by monotonicity $\varphi
(\widetilde{\theta})=2\pi.$ Since $\widetilde{u}|\partial D=\gamma\circ
\varphi$ exactly as in the proof of Lemma \ref{main} the alternative arises
that either $\varphi$ is continuous and $\widetilde{u}\left(  \partial
D\right)  =\Gamma$ or $\varphi$ is a step function with a jump of height
$2\pi.$ In the first case $\widetilde{u}$ is a solution to the asymptotic
Plateau problem in the $H_{2}^{1}-$sense and in the second one a punctured
minimal sphere has split off again. If the latter happens we can repeat the
whole blow-up process, in each step lowering the energy by a fixed amount, see
(\ref{des8}). This must stop as soon as the minimal area threshold (\ref{are})
were violated. This proves the theorem.
\end{proof}

\section{\label{SC}Proof of Theorem \ref{main2}}

\qquad We recall that a Hadamard manifold $N$ satisfies the strict convexity
condition if, given $x\in\partial_{\infty}N$ and a relatively open subset
$W\subset\partial_{\infty}N$ containing $x,$ there exists a $C^{2}$-open
subset $\Omega\subset\overline{N}$ such that $x\in\operatorname*{Int}\left(
\partial_{\infty}\Omega\right)  \subset W,$ where $\operatorname*{Int}\left(
\partial_{\infty}\Omega\right)  $ denotes the interior of $\partial_{\infty
}\Omega$ in $\partial_{\infty}N,$ and $N\setminus\Omega$ is convex. Loosely
speaking, this means that, as it happens with strictly convex bounded domains
in Euclidean spaces, we can take out a neighborhood in $\overline{N}$ at any
point at infinity of $N,$ which arbitrarily small asymptotic boundary, and
what remains is still convex.

Let $B$ be the model of $N$ as in Theorem \ref{main}. We only have to prove
that $\partial_{\infty}u(D)=\Gamma.$ Given $x\in\partial_{\infty}%
B\backslash\Gamma$ we prove that $x\notin\partial_{\infty}u(D)$ from what it
follows that $\partial_{\infty}u(D)\subset\Gamma.$ Since, by Theorem
\ref{main}, $u\in H_{2}^{1}\left(  D,\mathbb{R}^{n}\right)  $ it follows that
$\partial_{\infty}u(D)=\Gamma.$

Since $\Gamma$ is compact, there is $W\subset\partial_{\infty}B$ such that
$W\cap\Gamma=\varnothing.$ By the strict convex condition there is a $C^{2}$
convex neighborhood $\Omega$ of $N$ such that $x\in\operatorname*{Int}\left(
\partial_{\infty}\Omega\right)  \subset W.$ Let $d:\Omega\rightarrow\left[
0,+\infty\right)  $ be the distance to $\partial\Omega.$ Then the level
hypersurfaces of $d$ determine a foliation of $\Omega$ by equidistant
hypersurfaces to $\partial\Omega.$ From the Hessian Comparison Theorem if $S$
is a leaf of this foliation at a distance $d$ of $\partial\Omega$ then any
principal curvature $\lambda$ of $S$ with respect to the unit normal vector
field pointing to the connected component of $B\backslash S$ that does not
contain $x,$ satisfies $\lambda\geq a\tanh(ad).$ That is, the level
hypersurfaces of $d$ provides a foliation $\left\{  S_{d}\right\}  $ of
$\Omega$ which is convex towards the connected component of $B\backslash
S_{d}$ which does not contain $x.$ Since $\lim_{R\rightarrow\infty}\partial
u_{R}(D)=\Gamma$ for a sufficiently large $R_{0}$ we have $\partial
u_{R}(D)\cap\Omega=\emptyset$ for all $R\geq R_{0}.$ By the comparison theorem
it follows that $u_{R}(D)\cap\Omega=\emptyset$ for all $R\geq R_{0}$ and then
$x\notin\partial_{\infty}u(D),$ proving the theorem.

\end{document}